\documentclass[a4paper,11pt]{elsarticle}
\usepackage{amsmath}
\usepackage{amsthm}
\usepackage{amsfonts}
\usepackage{amssymb}
\usepackage{enumerate}
\usepackage{mathrsfs}
\usepackage{mathtools}
\usepackage{tikz}
\usepackage{float} 
\usepackage{xcolor}
\definecolor{gris}{rgb}{.5,.5,.5}

\usetikzlibrary{arrows}
\usetikzlibrary{decorations.pathreplacing}
\usepackage[pagebackref=true, bookmarksopen=true,colorlinks=true, linkcolor=red,citecolor=blue]{hyperref}

\usepackage[ruled]{algorithm2e}
\newtheorem{construction}{Construction}[section]
\newtheorem{theorem}[construction]{Theorem}

\newtheorem{corollary} [construction]{Corollary}

\newtheorem{definition} [construction]{Definition}

\newtheorem{lemma} [construction]{Lemma}

\newcommand{\cycles}{\operatorname{cycles}}
\newcommand{\Forest}{\operatorname{Forest}}
\newcommand{\Cycles}{\operatorname{Cycles}}

\begin{document}

\begin{abstract} 
    In the present work we prove that given any two unicycle graphs (pseudoforests)
    that share the same degree sequence there is a finite sequence of 2-switches
    transforming one into the other such that all the graphs in the sequence
    are also unicyclic graphs (pseudoforests).
\end{abstract}
\begin{keyword}
    2-switch\sep degree sequence\sep unicyclic graphs\sep pseudoforests\sep stability\sep Berge family.
    \MSC 05A15 \sep 05A19.
\end{keyword}

\begin{frontmatter}

    \title{2-switch transition on unicyclic graphs and pseudoforest}
    
    \author[daj]{Daniel A. Jaume\corref{cor1}}
    \ead{djaume@unsl.edu.ar}
    
    \author[daj]{Adr\'{\i}an Pastine}
    \ead{agpastine@unsl.edu.ar}
    
    \author[daj]{Victor N. Schv\"{o}llner}
    \ead{schvollner@unsl.edu.ar}
    \cortext[cor1]{Corresponding author: Daniel A. Jaume}
    
    \address[daj]{Departamento de Matem\'{a}ticas, \\
        Facultad de Ciencias F\'{\i}sico-Matem\'{a}ticas y Naturales,\\
        Universidad Nacional de San Luis,\\
        Instituto de Matem\'atica Aplicada San Luis (UNSL-CONICET),\\ 
        Universidad Nacional de San Luis,\\
        San Luis, Rep\'{u}blica Argentina.}
    \date{Received: date / Accepted: date}
\end{frontmatter}



%
%

\section{Introduction}\label{intro}
Every graph $G=(V,E)$ in the present article is finite, simple, undirected and labeled. We use \(|G|\) and $\|G\|$ to denote the order of \(G\)  and the size of $G$  respectively. In this work the set of vertices of $G$ is always a subset of $[n]:=\{1,\ldots  ,n\}$, for some $n$. When there may be ambiguity we use  \(V(G)\) and \(E(G)\) to denote the vertex set and the edge set of $G$, otherwise we just use $V$ and $E$. Vertex adjacency is denoted by $x\sim y$, and we denote the edge \(\{x,y\}\) by $xy$ (i.e., we say that $xy\in E$). The number of connected components of a graph \(G\) is denoted by  \(\kappa(G)\) and its set of components by \(\mathcal{K}(G)\). The subgraph of \(G\) obtained by deleting vertex \(v\) is denoted by \(G-v\). Similarly, \(G-e\) is the subgraph of \(G\) obtained by deleting edge \(e\), \(G+e\) or \(G+ab\) is the graph obtained by adding an edge to \(G\). If $W$ is a set of vertices (edges) of a graph $G$, $G-W$ denotes the subgraph obtained by deleting the vertices (edges) in $W$.

The degree sequence of a graph $G$ with vertex set $V(G)=[n]$ is the sequence $s(G)=(d_{1}, \dots, d_{n})$, where $d_i$ is the degree of vertex $i$. We assume that \(d_{1}\geq d_{2}\geq \cdots \geq d_{n}\).
A sequence $s=(d_{1}, \dots, d_{n})$ is graphical if there is a graph such that $s$ is its degree sequence.

Let \(s=(d_{1},\dots,d_{n})\) be a graphical sequence. By \(\mathcal{U}(s)\) we denote the set of all the unicyclic graphs  (connected graphs with just one cycle) with  degree sequence \(s\), and by \(\mathcal{P}(s)\) we denote the set of all the pseudoforests (graphs whose connected components are trees or unicyclic graphs) with  degree sequence \(s\).

The name pseudoforest is justified by analogy to the more commonly studied trees and forests. In \cite{Gabow@spanning1988} the authors attribute the study of pseudoforests to Dantzig's 1963 book on linear programming, see \cite{dantzig@linear1963}, in which pseudoforests arise in the solution of certain network flow problems. Pseudoforests also form graph-theoretic models of functions and occur in several algorithmic problems. Pseudoforests are sparse graphs (their number of edges is linearly bounded in terms of their number of vertices, in fact, they have at most as many edges as they have vertices)  and their matroid structure allows several other families of sparse graphs to be decomposed as unions of forests and pseudoforests. The name ``pseudoforest'' was first used in \cite{picard@network1982}.

One of the most studied problems in the literature is: given a graph parameter (clique number, domination number, matching number, etc.), finding the minimum and maximum values for the parameter in a family of graphs, see \cite{BockRat,GHR1,GHR2,wang2008extremal}.
Another interesting problem is deciding which values between the minimum and the maximum can be realized by a graph in the family, see \cite{kurnosov2020set,Rao,jaume20202switch}.

Let $G$ be a graph containing four distinct vertices $a,b,c,d$ such that $ab,cd\in E$ and $ac,bd\notin E$. The process of deleting the edges $ab$ and $cd$ from $G$ and adding $ac$ and $bd$ to $G$ is referred to as a 2-switch in $G$, this is a classical operation, see \cite{Berge}. Even though the order of the vertices matters, it is usual to talk about a 2-switch between the edges $\{a,b\}$ and $\{c,d\}$. If $G'$ is the graph obtained from $G$ by a 2-switch, it is straightforward to check that $G$ and $G'$ have the same degree sequence. In other words, this operation preserves the  degree sequence.

In \cite{jaume20202switch} the 2-switch was introduced as a function. 


\begin{definition}\textbf{(}\cite{jaume20202switch}\textbf{)}
	Let \(a,b,c,d \in [n]\) and let $G$ be a graph. The matrix ${{a \ b}\choose{c \ d}}$ is said to be \textbf{interchangeable} in $G$, if it satisfies the following conditions: 
	\begin{enumerate}
		\item $ab,cd\in E(G)$;
		\item $\{a,b\}\cap \{c,d\}=\varnothing$;
		\item $ac,bd\not\in E(G)$.
	\end{enumerate}
	Otherwise, ${{a \ b}\choose{c \ d}}$ is said to be \textbf{trivial} for $G$.	
\end{definition}
Notice in particular that if at least one of $a,b,c,d$ is not a vertex of $G$, then ${{a \ b}\choose{c \ d}}$ is trivial for $G$.

\begin{definition}\textbf{(}\cite{jaume20202switch}\textbf{)}
	\label{def2switch}
	Let \(n\) be a  integer and \(a,b,c,d \in [n]\), $A={{a \ b}\choose{c \ d}}$ and $G$ a graph. A \textbf{2-switch} is a function $\tau_{A}: \mathcal{G}\rightarrow \mathcal{G}$, where \(\mathcal{G}\) is the set of all graphs, defined as follows:
	\begin{equation}
		\tau_{A}(G)= \left\{ 
		\begin{array}{ll}
			G-ab-cd+ac+bd, &   \text{ if } A \text{ is interchangeable in } G, \\
			\\ G, & \text{ if } A \text{ is trivial for } G. \\
		\end{array}
		\right.
	\end{equation}
	
	If $\tau_{A}(G)=G$, we say that $\tau_{A}$ is \textbf{trivial} for $G$. The matrix $A$ is said to be an \textbf{action matrix} of $\tau_{A}$. 
\end{definition}

An important fact about degree sequences is that, given two graphs with the same degree sequence, one can be obtained from the other by applying successive 2-switches.

\begin{theorem}\label{BergeTeo}
    If $G,H\in \mathcal{G}(s)$, there exists a 2-switch sequence transforming $G$ into $H$. 
\end{theorem}

Theorem \ref{BergeTeo} appears throughout the literature, although its earliest reference appears  most likely in \cite{Berge}.

In \cite{BockRat} the authors study the matching number of trees with a given degree sequence. The authors find minimum and maximum values 
for the matching number in this family, and then show that every value between the minimum and the maximum
is realized by a graph in the family. When this happens for a parameter, it is said
to have \textbf{interval property} with respect to the family of graphs being studied.

A family of graphs is said to be a \textbf{Berge family}, if for all pair of graphs of the family, there exist a finite sequence of 2-switch that transform one into the other, such that every intermediate graph is a member of the family. In \cite{jaume20202switch} it was proved that trees and forests with the same degree sequence are Berge families, they also introduce two special types of 2-switch.	A nontrivial 2-switch $\tau$ over a tree $T$ is said to be a \textbf{t-switch} if $\tau(T)$ is a tree.  A nontrivial 2-switch $\tau$ over a forest $F$ is said to be an \textbf{f-switch} if $\tau(F)$ is a forest. They also proved the following theorem that shows that trees and forests with the same degree sequence are Berge families.

\begin{theorem}[Forest Transition Theorem, \cite{jaume20202switch}] \label{TTT}
	Let $F$ and $F'$ be two forests (trees) with the same degree sequence. Then \(F\) can be transformed into \(F'\) by a finite sequence of  f-switches (t-switches).
\end{theorem}

A graph parameter \(\xi\) is said to be \textbf{stable} under $2$-switch, see definition 5.1 in \cite{jaume20202switch}, if given \(G\) a graph and \(\tau\) a $2$-switch, then
\[
\left| \xi\left(\tau(G)\right)-\xi(G)\right| \leq 1.
\]
\begin{theorem}[\cite{jaume20202switch}]
	The following parameters are stable under 2-switch:
	\begin{enumerate}
		\item matching number, 
		\item independence number,
		\item domination number,
		\item path-covering number,
		\item edge-covering number,
		\item vertex-covering number,
		\item chromatic number,
		\item clique number,
		\item number of connected components.
	\end{enumerate}
\end{theorem}

Note that any stable discrete graph parameter has the interval property with respect to every Berge family.

%
%

\section{u-switch}

In this section we introduce and characterize a particular type of 2-switch, the u-switch, which is closed over unicyclic graphs. We start with some technical results.

A graph $G$ is said to be \textbf{unicyclic} if it is connected and contains exactly one cycle, i.e. \(\left\| G\right\|=|G|\). If $s$ is a graphical sequence, we denote by $\mathcal{U}(s)$ the set of all unicyclic graphs with degree sequence $s$. We assume that \(s\) is the degree sequence of at least one unicyclic graph.

A 2-switch \(\tau\) is said to be a \textbf{breaker} over a graph \(G\) whenever \(\kappa(\tau(G))>\kappa(G)\). If \(C\) is a cycle graph, and \(\tau\) is a breaker over \(C\), then \(\tau(C)\) is the union of two disjoint cycles. If \(\tau\) is not a breaker over \(C\), then \(\tau(C)\) is a cycle isomorphic to \(C\).

The following observation will be used many times. Let $\tau= {{a \ b}\choose{c \ d}}$ be a 2-switch and $e$ an edge of $G$, \(e \notin \{ab,\,cd\}\). If $e\notin\tau(G-e)$, then $\tau(G-e)+e=\tau(G)$. In particular, if $\tau(G-e)$ is a tree, then $\tau(G)$ is unicyclic.

We define two edge-disjoint subgraphs of a graph \(G\): \(\Cycles(G)\) and \(\Forest(G)\). By \(\Cycles(G)\) we denote the subgraph induced by all
the vertices of \(G\) such that they are in a cycle of \(G\). By \(\Forest(G)\) we denote the forest
that remains after taking away from \(G\) all the edges in \(\Cycles(G)\), i.e. \(\Forest(G) := G - E(\Cycles(G))\), i.e.
this subgraph is the forest whose components are trees (possibly trivial) attached to vertices of \(\Cycles(G)\). Note that if \(G\) is just a set of cycles, then \(\Forest(G)\) is a set of isolated vertices. Clearly  $V(\Cycles(G))\subseteq V(G)=V(\Forest(G))$. Observe that the edge set $E(G)$ can always be written as the disjoint union of $E(\Cycles(G))$ and $E(\Forest(G))$.

\begin{definition}
    A nontrivial 2-switch $\tau$ over a unicyclic graph $U$ is said to be a \textbf{u-switch} if $\tau(U)$ is unicyclic.
\end{definition}

\begin{lemma}
    \label{lem.tswitch.U-e}
    
    Let $U\in \mathcal{U}(s)$ and let $e$ be any edge of $\Cycles(U)$. If $\tau$ is a t-switch in $U-e$ between the edges $ab,cd$ of $\Forest(U)$, then $e\notin\tau(U-e)$.
\end{lemma}
\begin{proof}
	Let $C$ be the cycle formed by the edges in $\Cycles(U)$.
	Notice that $C-e$ is a subgraph of both $U-e$ and $\tau(U-e)$, because
	none of the edges in $C-e$ are involved in $\tau$.
	If $e\in \tau(U-e)$, then $\tau(U-e)$ contains the cycle $C$, contradicting
	the fact that $\tau$ is a t-switch in $U-e$. Hence, $e\not\in \tau(U-e)$.
\end{proof}

\begin{lemma}
    \label{lem2.tswitch.U-e}
    
    Let $U\in \mathcal{U}(s)$ and consider the following edges: $ab\in \Cycles(U),cd\in \Forest(U)$. If $\tau$ is a 2-switch between $ab$ and $cd$, then there exists an edge $e\in \Cycles(U)-ab$ such that $e\notin\tau(U-e)$. 
\end{lemma}   

\begin{proof}
	Let $\tau={{a \ b}\choose{c \ d}}$.
	As $cd\in E(\Forest(U))$, at least one of $c$ and $d$ is not a vertex of $Cycles(U)$.
	Assume without loss of generality that $d\not\in V(Cycles(U))$.
	Let $v\neq a$ be a neighbor of $b$ in $\Cycles(U)$, and let $e=bv$.
	Notice that $v\neq d$ because $d\not\in V(\Cycles(U))$. Therefore, $e\not\in \tau(U-e)=U-e-ab-cd+ac+bd$.
%
%
\end{proof}
The following theorem characterizes u-switches. 
\begin{theorem}
	\label{uswitchcaract}
	
	Let $\tau$ be a nontrivial 2-switch between two disjoint edges $e_{1}$ and $e_{2}$ of a unicyclic graph $U\in \mathcal{U}(s)$. Then the following statements hold: 
	\begin{enumerate}
		\item If $e_{1},e_{2}\in E(\Forest(U))$, $e\in E(\Cycles(U))$, then \(\tau\) is a u-switch over $U$ if and only if $\tau$ is a t-switch over $U-e$.
		\item If $e_{1}\in E(\Cycles(U))$ and $e_{2}\in E(\Forest(U))$, then \(\tau\) is a u-switch over $U$.
		\item If $e_{1},e_{2}\in E(\Cycles(U))$, then \(\tau\) is a u-switch over $U$ if and only if $\tau$ is not a breaker over $\Cycles(U)$.		
	\end{enumerate}
\end{theorem}

\begin{proof} 
	(i \(\Leftarrow\) ) By Lemma \ref{lem.tswitch.U-e} \(e \notin \tau(U-e)\). Since \(\tau(U-e)\) is a tree, \(\tau(U-e)+e=\tau(U)\) is a unicyclic graph.
	 
	(i \(\Rightarrow\) )  If $e_{1},e_{2}\in E(\Forest(U))$, $e\in E(\Cycles(U))$, and $\tau$  is not a t-switch in $U-e$, then \(\kappa(\tau(U))=2\).
	
	(ii \(\Rightarrow\) ) By Lemma \ref{lem2.tswitch.U-e}  there exists \(e \in \Cycles(U)-e_{1}\) such that \(e\notin\tau(U-e)\). Clearly, $U-e$ is a tree. If $\tau(U-e)$ is a tree, then \(\tau(U-e)+e=\tau(U)\) is a unicyclic graph. Otherwise, $\tau(U-e)$ is the union of a unicyclic graph $U'$ and a tree $T$. Since $e$ links $U'$ to $T$, we have that $\tau(U-e)+e=\tau(U)$ is a unicyclic graph.
	
	
	(iii \(\Rightarrow\) ) If $\tau$ is nontrivial but a breaker over $\Cycles(U)$, then $\tau(U)$ has two components, Therefore $\tau(U)\notin \mathcal{U}(s)$. 
	
	(iii \(\Leftarrow\) ) If $\tau$ is not a breaker over $\Cycles(U)$, then $\tau(U)$ is a unicyclic graph, because \(\tau(C)\) is a cycle isomorphic to \(C\).
\end{proof}

%
%
\section{Unicyclic graphs with a given degree sequence are a Berge family}

In order to prove that the families of unicyclic graphs with a given degree sequence are Berge families, see Theorem \ref{unicTransition}, we need some technical lemmas.

\begin{lemma}
	\label{sharedleaf.lem1} \label{sharedleaf.lem2}
	
	Let $U$ be a unicyclic graph such that $\ell u\in E(U)$, $d_{\ell}=1$ and $u\notin V(\Cycles(U))$. 
	\begin{enumerate}
		\item If $v\in V(\Cycles(U))$, then there exists an u-switch $\tau$ over $U$ such that $\ell v\in E(\tau(U))$.    
		\item If $v\notin V(\Cycles(U))$, $v\neq u$, and $d_{v}\geq 2$, then there exist a u-switch $\tau$ over $U$ such that $\ell v\in E(\tau(U))$.
	\end{enumerate}
	
\end{lemma}
\begin{proof}
	 (i) Since $d_{v}\geq 2$, there exists a vertex $w\neq u$ such that $vw\in \Cycles(U)$. Since $u\notin V(\Cycles(U))$, we can choose $w$ such that $uw \notin E(U)$. Thus, $uw$ and $\ell v$ are not edges of $U$. Moreover, $\ell u\cap vw=\varnothing$. Hence, $\tau:={{\ell \ u}\choose{v \ w}}$ is a nontrivial 2-switch over $U$ such that $\ell v\in\tau(U)$. By Theorem \ref{uswitchcaract}, $\tau(U)$ is a unicyclic graph.
	 
	 (ii)  Since $d_{v}\geq 2$, there exists a vertex $w$, adjacent to $v$, such that $w\neq u$. Let $T_{u}$ be the component of $\Forest(U)$ that contains $u$. We have two possibilities: 1) $v\in T_{u}$; 2) $v\notin T_{u}$. 
	 
	 Suppose $v\in T_{u}$. Since $u,v\notin V(\Cycles(U))$, we can choose $w$ in such a way that $vw\in E(T_{u})$. So, there is a unique path from $\ell$ to $w$ and it has the form $\ell u...vw$. If $\tau:={{\ell \ u}\choose{v \ w}}$, then $\tau$ is a t-switch in $T_{u}$ such that $\ell v\in E(\tau(U))$. Moreover, $\tau$ is clearly a u-switch over $U$. 
	 
	 If $v\notin T_{u}$, we can choose $w$ in such a way thar the two paths from $\ell$ to $w$ have the form $\ell u...vw$. Therefore, $\tau:={{\ell \ u}\choose{v \ w}}$ is a nontrivial 2-switch in $U$ such that $\ell v \in E(\tau(U))$. By Theorem \ref{uswitchcaract}, $\tau$ is actually a u-switch over $U$.
\end{proof}

\begin{corollary}
	\label{sharedleaf.lem3}
	
	Let $U,U'\in\mathcal{U}(s)$ such that $\ell u\in E(U)$,  $\ell v\in E(U')$, $d_{\ell}=1$, and $v\in V(\Cycles(U))$. If $d_{u}=2$, then there exists a u-switch $\tau'$ over $U'$ such that $\ell v\in E(\tau'(U'))$.
\end{corollary}
\begin{proof}
	Since $d_{v}=2$ and $\ell v\in E(U')$, we have that $v\notin V(\Cycles(U'))$. Since also $d_{u}\geq 3$, we can use Lemma \ref{sharedleaf.lem1} on $U'$.
\end{proof}

\begin{lemma}	\label{sharedleaf.lem4} 	\label{sharedleaf.lem5}
	Let $U$ be a unicyclic graph such that $\ell u\in E(U)$, $d_{\ell}=1$, and $u\in V(\Cycles(U))$. 
	\begin{enumerate}
		\item If $v\in V(\Cycles(U))$, $v \neq u$, and $d_{v}>2$, then there exists a u-switch $\tau$ over $U$ such that $\ell v\in E(\tau(U))$.
		\item If $v\notin V(\Cycles(U))$ and $d_{v}\geq 2$, then there exists a u-switch $\tau$ over $U$ such that $\ell v\in E(\tau(U))$.
	\end{enumerate}
\end{lemma}

\begin{proof}
	(i) Since $v\in V(\Cycles(U))-u$ and $d_{v}\geq 3$, we can choose a vertex $w$ such that $vw \in E(\Forest(U))$. Therefore, the two paths in $U$ from $\ell$ to $w$ have the form $\ell u...vw$. Hence, by Theorem \ref{uswitchcaract}, $\tau:={{\ell \ u}\choose{v \ w}}$ is a u-switch over $U$ such that $\ell v\in E(\tau(U))$.
	
	(ii) Since $v\notin V(\Cycles(U))$ and $d_{v}\geq 2$, we can choose a vertex $w\notin V(\Cycles(U))$ in such a way that the path/s (depending on if $v$ lies in the same component of $u$ respect to $\Forest(U)$) from $\ell$ to $w$ has/have the form $\ell u...vw$. Hence, by Theorem \ref{uswitchcaract}, $\tau:={{\ell \ u}\choose{v \ w}}$ is a u-switch over $U$ such that $\ell v\in E(\tau(U))$.
\end{proof}

\begin{definition}
	Let $G,H$ be two graphs. A vertex $\ell\in V(G)\cap V(H)$ is said to be a \textbf{shared leaf} by $G$ and $H$ if $\deg_{G}(\ell)=\deg_{H}(\ell)=1$ and $\ell v\in E(G)\cap E(H)$ for some vertex $v$.    
\end{definition}

\begin{theorem}
	\label{shared.leaf}
	
	Let $U,U'\in\mathcal{U}(s)$. If $U$ and $U'$ are not cycles and do not share leaves, then there exists a u-switch $\tau$ over $U$ such that $\tau(U)$ shares a leaf with $U'$.
\end{theorem}

\begin{proof}
	Since $U,U'\in\mathcal{U}(s)$ but they are not cycles, they have at least one leaf $\ell$. Since they do not share leaves, we have $\ell u\in E(U)$ and $\ell v\in E(U')$, with $u\neq v$ and $d_{v},d_{u}\geq 2$. Now, depending on where vertices $u$ and $v$ are in $U$, we apply Lemma \ref{sharedleaf.lem1}, or Corollary  \ref{sharedleaf.lem3}, or Lemma \ref{sharedleaf.lem4}.
\end{proof}

Note that the empty sequence of 2-switches , denoted by \((\varnothing)\), is an f-switch sequence and also a u-switch sequence.

\begin{lemma}
	\label{shared.edge.cycles.1}
	
	Let $C, C'$ be two cycles with the same degree sequence. If $e\in E(C)\cap E(C')$, then there exists a u-switch sequence transforming $C$ into $C'$.
\end{lemma}

\begin{proof}
	Since $C-e$ and $C'-e$ are two trees with the same degree sequence, by the Forest Transition Theorem \ref{TTT},  there is a t-switch sequence $(\tau_{i})_{1\leq i \leq n}$ transforming $C-e$ into $C'-e$. Since $C-e$ is a path, for $i=1,\, \dots, \, n$ the tree $T_{i}$ of the transition is a path too. Moreover,  for $i=1,\, \dots, \, n$ we  have that $e\notin E(T_{i})$. Thus, $(\tau_{i})_{1\leq i \leq n}$ is a u-switch sequence from $C$ to $C'$.
\end{proof}

\begin{lemma}
	\label{shared.edge.cycles.2}
	
	Let $C$ be a cycle. If $u$ and $v$ are non-adjacent vertices of $C$, then there exists a u-switch $\tau$ over $C$ such that $uv\in E(\tau(C))$.
\end{lemma}

\begin{proof} Clearly, $|C|>3$. Choose a neighbor of $u$, call it $x$. This choice determines a unique path $P$ in $C$, from $u$ to $v$, passing through $x$. Let $w$ be the neighbor of \(v\) outside of $P$. By Theorem \ref{uswitchcaract}, $\tau:={{u \ x}\choose{v \ w}}$ is the required u-switch.
\end{proof}

\begin{theorem}
	\label{cycle.transition}
	
	If $C$ and $C'$ are two cycles with the same degree sequence, then there exists a u-switch sequence transforming $C$ into $C'$. 
\end{theorem}  

\begin{proof} If $E(C)\cap E(C')\neq \varnothing$, use Lemma \ref{shared.edge.cycles.1}. Otherwise, use first Lemma \ref{shared.edge.cycles.2} and then Lemma \ref{shared.edge.cycles.1}.
\end{proof}

The next Theorem says that all the unicyclic graphs with a given degree sequence are a Berge family.

\begin{theorem}
	\label{unicTransition}
	
	If $U,U'\in \mathcal{U}(s)$, then there exists a u-switch sequence transforming $U$ into $U'$.
\end{theorem} 

\begin{proof} If $U=U'$, $(\varnothing)$ is the required sequence. Assume $U\neq U'$. If $U$ and $U'$ are cycle graphs, use Theorem \ref{cycle.transition}. Otherwise, notice that $U$ and $U'$ have at least one leaf. For this part of the proof, we proceed by induction on $n=|U|=|U'|$. For $n\leq 4$ the statement is trivially true. Let $n>4$ and suppose the statement holds for any two unicyclic graphs of order less than $n$ with the same degree sequence. 

If $U$ shares a leaf $\ell$ with $U'$, then $U-\ell$ and $U'-\ell$ are two unicyclic graphs of order $n-1$ with the same degree sequence. Thus, the inductive hypothesis applies: there exists a u-switch sequence transforming $U-\ell$ into $U'-\ell$. Hence, the same sequence is a u-switch sequence transforming $U$ into $U'$ too.

If $U$ and $U'$ do not share leaves, by Theorem \ref{shared.leaf}, there exists a u-switch \(\tau\) such that \(\tau(U)\) and \(U'\) share a leaf. 
\end{proof}

%
%

\section{p-switch}
Let $G$ be a graph, by $\cycles(G)$ we denote the number of subgraphs of $G$ isomorphic to a cycle. The number $c(G):=\max_{H\in \mathcal{K} (G)}\{\cycles(H)\}$ is called \textbf{cyclicity} of \(G\). In other words, $G$ contains at most $c(G)$ cycles in each of its components.

\begin{definition}
	\label{pseudodef}
	
	A \textbf{pseudoforest} is a graph $G$ such that $c(G)\in \{0,1\}$.
\end{definition}
The set of all forests with degree sequence $s$ is denoted by $\mathcal{F}(s)$. The set of all pseudoforests with degree sequence $s$ is denoted by $\mathcal{P}(s)$. Clearly: $\mathcal{U}(s)\subseteq \mathcal{P}(s)$. These two families, $\mathcal{U}(s)$ and $\mathcal{P}(s)$, do not behave ``well'' like trees and forests: recall that if $T\in \mathcal{F}(s)$ is a tree, then all members of $\mathcal{F}(s)$ are trees too. In fact, if $U\in \mathcal{P}(s)$ is unicyclic, then the rest of pseudoforests in $\mathcal{P}(s)$ need not be all unicyclic.

\begin{definition}
	\label{pswitchdef}
	
	A nontrivial 2-switch $\tau$ over a pseudoforest $G$ is said to be a \textbf{p-switch} if $\tau(G)$ is a pseudoforest.
\end{definition}

Notice that t-switches, f-switches and u-switches are clearly particular cases of p-switches. 

Given two graphs \(H\) and \(G\) such that \(V(G)\cap V(H)=\varnothing\), by \(G\dot{\cup}H \) we denote the disjoint union of both, i.e. the graph with vertex set \(V(G) \cup V(H)\) and edge set \(E(G) \cup E(G)\).

\begin{lemma}
	\label{lem2switchpswitch}
	Let $F$ be a forest and let $U$ be a unicyclic graph vertex-disjoint from \(F\). We have the following:
	\begin{enumerate}
		\item Every nontrivial 2-switch over $F$ is a p-switch.
		\item Let $\tau= {{a \ b}\choose{c \ d}}$ be a 2-switch over $G=F\dot{\cup}U$. If $ab\in E(F)$ and $cd\in E(\Forest(U))$, then $\tau$ is a p-switch in $G$.
	\end{enumerate}
\end{lemma}

\begin{proof}
	Statement (i) holds because the number of connected components is stable. For (ii), choose any $e\in E(\Cycles(U))$ and notice that $G-e$ is a forest. Since $ab$ and $cd$ are in different components, $\tau(G-e)$ is a forest and so $\tau(G-e)+e$ contains at most one cycle. Since $|e\cap \{a,b,c,d\}|\in \{0,1\}$, $e\notin\tau(G-e)$. Hence, $\tau(G-e)+e=\tau(G)$.
\end{proof}
The next Lemma is a direct consequence of Theorem \ref{uswitchcaract}.
\begin{lemma}
	\label{2switch.in.C_U=pswitch.in.U}

	If $U\in \mathcal{U}(s)$, then every nontrivial 2-switch between two edges of $\Cycles(U)$ is a p-switch over $U$.
\end{lemma}

\begin{lemma}
	\label{lema2switchpswitch}
	
	Let $\tau= {{a \ b}\choose{c \ d}}$ be a nontrivial 2-switch on $G\in \mathcal{P}(s)$. Suppose that one of the following conditions holds: 
	\begin{enumerate}
		\item $ab\in E(\Forest(G))$ and $cd\in E(\Cycles(G))$;
		\item $ab,cd\in E(\Cycles(G))$.
	\end{enumerate}
	Then, $\tau$ is p-switch over $G$. 
\end{lemma}

\begin{proof}
	For each case of the hypothesis we have the following subcases: (A) $ab$ and $cd$ lie in the same component $H$ of $G$; (B) $ab$ and $cd$ lie in different components of $G$. 

	For (i.A) use respectively Theorem \ref{uswitchcaract} and for (ii.A) use Lemma \ref{2switch.in.C_U=pswitch.in.U}. For (i.B), observe that $\tau$ breaks the cycle that contains $cd$, and thus $c(\tau(G))\leq c(G)$. 	Finally, for (ii.B), note that $\tau$ glues the two cycles containing $ab$ and $cd$ together in a new cycle. Hence, $c(\tau(G))=c(G)$.
\end{proof}

\begin{lemma}
	\label{lemma2pswitchFUFU'}
	
	Let $\tau = {{a \ b}\choose{c \ d}}$ be a 2-switch on a pseudoforest $G$ and let $U$ and $U'$ be different unicyclic components of $G$ such that $ab\in \Forest(U)$ and $cd\in \Forest(U')$. Then, $\tau$ is a p-switch in $G$ if and only if $\tau'={{a \ b}\choose{d \ c}}$ is not. 
\end{lemma}

\begin{proof}
	The proof is straightforward. 
\end{proof}

The next Theorem characterizes when a 2-switch over a pseudoforest gives another pseudoforest.

\begin{theorem}
	\label{pswitchcaract}
	
	Let $\tau= {{a \ b}\choose{c \ d}}$ be a nontrivial 2-switch on a pseudoforest $G$. Then the following statements hold.
	\begin{enumerate}
		\item If \(ab\) and \(cd\) are in different components of \(\Forest(U)\) for some unicyclic component \(U\) of \(G\)
		and \(e \in E(\Cycles(U))\), then \(\tau\) is a p-switch if and only if it is a t-switch on \( U-e\).
		
		\item If \( ab \in E(\Forest(U))\) and \(cd \in E(\Forest(U'))\) for two different unicyclic components \(U\) and \( U'\) of \(G\),
		then \(\tau\) is a p-switch if and only if ${{a \ b}\choose{d \ c}}$ is not a p-switch.
		\item In any other case \(\tau\) is a p-switch.
	\end{enumerate}
\end{theorem} 

\begin{proof}
	If (i) holds, then $\tau$ is a u-switch in $U$ by Theorem \ref{uswitchcaract}. Hence, it is a p-switch over $G$. 
	
	If $ab\in  E(\Forest(U))$, $cd\in E(\Forest(U'))$ and ${{a \ b}\choose{d \ c}}$ is not a p-switch, then $\tau$ is a p-switch by Lemma \ref{lemma2pswitchFUFU'}.
	
	In order to check that $\tau$ is a p-switch in $G$ in any other case apply Lemmas \ref{lem2switchpswitch} and \ref{lema2switchpswitch}. 

	For the converse, suppose that $\tau$ does not operate as described above. If $\tau$ is not a t-switch in $U-e$ when $ab,cd$ are in different components of $\Forest(U)$ and $e\in E(\Cycles(U))$, it is easy to see that $c(\tau(U))\geq 2$. Hence, $\tau$ is not a p-switch in $G$. If ${{a \ b}\choose{d \ c}}$ is a p-switch in $G$, when $ab\in E(\Forest(U))$ and $cd\in E(\Forest(U'))$, then by Lemma \ref{lemma2pswitchFUFU'} $\tau$ is not a p-switch in $G$.	
\end{proof}

%
%

\section{Pseudoforest graphs with a given degree sequence are a Berge family}

\begin{lemma}
	\label{pseudosize}
	
	If $G$ is a pseudoforest, then $||G||+\kappa(G)=|G|+\cycles(G)$.
\end{lemma}

\begin{proof}
If we remove an edge from every cycle of $G$, then we obtain a generating forest $F\leq G$ such that $\kappa(F)=\kappa(G)$. Therefore, $\left\|F\right\|=|G|-\kappa(G)$. On the other hand, $\left\|F \right\|=\left\|G \right\|-\cycles(G)$ and hence $\left\|G \right\|-\cycles(G)=|G|-\kappa(G)$.
\end{proof}
\begin{lemma}
	The function $\zeta :\mathcal{P}(s)\rightarrow \mathbb{Z}$ defined by $\zeta(G):= \kappa(G)-\cycles(G)$ is a non-negative constant.
\end{lemma}

\begin{proof}
	By Lemma \ref{pseudosize}, we have $|G|-||G||=\kappa(G)-\cycles(G)=\zeta(G)$. By Definition \ref{pseudodef}, $\zeta\geq 0$. Since all members of $\mathcal{P}(s)$ have the same order and size, $\zeta$ is constant.
\end{proof}


\begin{corollary}\label{coro_ciclos}
	If $G,H\in \mathcal{P}(s)$, then $\cycles(G)=\kappa(G)$ if and only if $\cycles(H)=\kappa(H)$.
\end{corollary}

%
%
\begin{corollary}\label{coro_ciclos2}
	Let $G$ and $H$ be pseudoforests such that $\cycles(G)=\kappa(G)$ and $\cycles(H)<\kappa(H)$. Then, there is no sequence of 2-switches transforming $G$ into $H$.
\end{corollary}

%

\begin{lemma}
	\label{pseudo>>>unic}
	
	Every pseudoforest $G$ with $\cycles(G)=\kappa(G)$ can be transformed into a unicyclic graph by a sequence of p-switches.
\end{lemma}  

\begin{proof}
	If $G$ is connected, we are done. If $\kappa(G)\geq 2$, then observe that we can ``glue together'' 2 components $U,U'$ of $G$ by performing a 2-switch $\tau$ between $e_{1}\in E(\Cycles(U))$ and $e_{2}\in E(\Cycles(U'))$. By Theorem \ref{pswitchcaract}, we know that $\tau$ is a p-switch in $G$, and by the proof of Lemma \ref{lema2switchpswitch} we know that $\tau(U\dot{\cup}U')$ is a unicyclic graph. Now, $\kappa(\tau(G))=\kappa(G)-1$. Therefore, we repeat the process until we obtain a connected pseudoforest $H$. By Corollary \ref{coro_ciclos}, $\cycles(H)=\kappa(H)=1$. Thus, $H$ is a unicyclic graph.
\end{proof}

\begin{lemma}
	\label{pseudo>>>forest}
	
	Every pseudoforest $G$ with $\cycles(G)<\kappa(G)$ can be transformed into a forest by a sequence of p-switches.
\end{lemma}  

\begin{proof}
	Every pseudoforest $G$ with $\cycles(G)<\kappa(G)$ can be written as $G=H\dot{\cup}F$, where $H$ is a pseudoforest such that each of its components is a unicyclic graph, i.e. $\cycles(H)=\kappa(H)$, and $F$ is a forest. Then, we can apply Lemma \ref{pseudo>>>unic} to $H$ to obtain from $G$ a pseudoforest $G'=U\dot{\cup}F$, where $U$ is a unicyclic graph. Now, perform a 2-switch $\tau$ between $e_{1}\in E(F)$ and $e_{2}\in E(\Cycles(U))$. Then, $\tau$ is a p-switch by Theorem \ref{pswitchcaract} and $\tau(G')$ is a forest by the proof of Lemma \ref{lema2switchpswitch}.
\end{proof}

We can now prove that the family of unicyclic graphs of a given degree sequence is a Berge family.
\begin{theorem}
	\label{pseudoTransition}
	
	If $G,H\in \mathcal{P}(s)$, then there exists a sequence of p-switches transforming $G$ into $H$.
\end{theorem} 
\begin{proof} If $\cycles(G)=\kappa(G)$, then $\cycles(H)=\kappa(H)$ too, by Corollary \ref{coro_ciclos}. Now, apply Lemma \ref{pseudo>>>unic} to $G$ and $H$ to obtain respectively $U,U'\in \mathcal{U}(s)$. By Theorem \ref{unicTransition}, we can transform $U$ into $U'$ by a sequence of u-switches and hence we can transform $G$ into $H$ by a sequence of p-switches.

If $\cycles(G)<\kappa(G)$, then by Corollary \ref{coro_ciclos2} $\cycles(H)<\kappa(H)$. Now, apply Lemma \ref{pseudo>>>forest} to $G$ and $H$ to obtain respectively $F,F'\in \mathcal{F}(s)$. By the Forest Transition Theorem, we can transform $F$ into $F'$ by a sequence of f-switches and hence we can transform $G$ into $H$ by a sequence of p-switches.\end{proof}


Concluding remark: The families of graphs \(\mathcal{U} (s)\) and \(\mathcal{P} (s)\) are Berge families. Therefore, for any discrete graph parameter 2-switch stable, as matching number, independence number, domination number, path-covering number, edge-covering number, vertex-covering number, chromatic number, clique number, number of connected components, we have that for any \[k \in \left[ \min_{U \in \mathcal{U} (S)} \mathtt{parameter} (U),\, \max_{U \in \mathcal{U}(S)} \mathtt{parameter} (U) \right] \] there exists \(U_{k} \in \mathcal{U} (s)\) such that  \(\mathtt{parameter} (U_{k})=k\). A similar result holds for
pseudoforests in \(\mathcal{P} (s)\).

\section{Acknowledgments}

This work was partially supported by the Universidad Nacional de San Luis, grants PROICO 03-0918 and PROIPRO 03-1720, and by MATH AmSud, grant 21-MATH-05.


\end{document}